\documentclass[12pt]{amsart}
\usepackage{amsmath}
\usepackage{amsfonts}
\usepackage{amssymb}
\usepackage{epsfig}
\usepackage{epstopdf}
\usepackage{xcolor}
\usepackage{verbatim}
 \usepackage[all]{xy}                                                                   
\usepackage{microtype}

\usepackage[alphabetic]{amsrefs}
\usepackage[margin=1in]{geometry}
\usepackage[colorlinks=true,urlcolor=blue,linkcolor=blue,citecolor=purple]{hyperref}
\usepackage{parskip}

\usepackage{pgf,tikz}
\usepackage{mathrsfs}
\usetikzlibrary{arrows}

\newtheorem*{theorem*}{Theorem}
\newtheorem{theorem}{Theorem}
  \newtheorem{lemma}[theorem]{Lemma}
  \newtheorem{claim}[theorem]{Claim}

   \newtheorem{corollary}[theorem]{Corollary}
  \newtheorem{proposition}[theorem]{Proposition}

  \theoremstyle{definition}
  
  \newtheorem{example}[theorem]{Example}
   \newtheorem{remark}[theorem]{Remark}

\newcommand{\RR}{\mathbb R}
\newcommand{\QQ}{\mathbb Q}
\newcommand{\ZZ}{\mathbb Z}

\newcommand{\cG}{\mathcal G}

\DeclareMathOperator{\trop}{trop}
\DeclareMathOperator{\val}{val}

\DeclareMathOperator{\spann}{span}
\DeclareMathOperator{\rk}{rk}
\DeclareMathOperator{\cl}{cl}

\DeclareMathOperator{\pos}{pos}
\DeclareMathOperator{\Gr}{Gr}

 \title[]{Higher Connectivity of Tropicalizations}
\author{Diane Maclagan}
\address{Mathematics Institute, University of Warwick, Coventry, CV4 7AL, United Kingdom}
\email{D.Maclagan@warwick.ac.uk}
 \author{Josephine Yu}
\address{School of Mathematics, Georgia Tech, Atlanta GA 30332, USA}
\email {jyu@math.gatech.edu}

\date{\today}

\begin{document}

  \begin{abstract}
 We show that the tropicalization of an irreducible $d$-dimensional
 variety over a field of characteristic $0$ is $(d-\ell)$-connected
 through codimension one, where $\ell$ is the dimension of the
 lineality space of the tropicalization.  From this we obtain a higher
 connectivity result for skeleta of rational polytopes.  We also prove
 a tropical analogue of the Bertini Theorem: the intersection of the
 tropicalization of an irreducible variety with a generic hyperplane
 is again the tropicalization of an irreducible variety.
 \end{abstract}

 \maketitle

\section{Introduction}

The tropicalization of a $d$-dimensional irreducible subvariety of the
algebraic torus $(K^*)^n$, where $K$ is a field, is the support of a
pure, $d$-dimensional polyhedral complex in $\RR^n$.  Being {\em pure}
means that all inclusion-maximal faces have the same dimension.  
The {\em facets} and {\em ridges} of a pure, $d$-dimensional polyhedral complex are the faces of dimensions $d$ and $d-1$ respectively.
A fundamental result in tropical geometry is that the tropicalization
of an irreducible variety is connected through codimension one: given any two
facets there is a path between them that goes through facets and
ridges only \cites{BJSST,  CartwrightPayne},
\cite{TropicalBook}*{Theorem 3.5.1}.  

In this paper we show the stronger result that the tropicalization of
an irreducible variety remains connected through codimension one even
after we remove some closed facets.  A pure polyhedral complex
is {\em $k$-connected through codimension one} if removing any
\mbox{$k-1$} closed facets leaves it connected through codimension
one.  In other words, the facet-ridge incidence hypergraph, whose
vertices are facets of the polyhedral complex and whose hyperedges are
the ridges, remains connected after any $k-1$ vertices and their
incident hyperedges are removed.  A hypergraph is connected if there
is a path from any vertex to any other vertex where each step connects
two vertices in the same hyperedge.

\begin{theorem} \label{t:maintheorem}
Let $K$ be a field of characteristic $0$ that is either algebraically
closed, complete, or real closed with convex valuation ring.  Let $X$
be a $d$-dimensional irreducible subvariety of $(K^*)^n$.  
Let $\Sigma$ be a pure $d$-dimensional rational polyhedral complex with support $|\Sigma|=\trop(X)$.  Write $\ell$ for the dimension of the lineality space of $\Sigma$. 
Then $\Sigma$ is $(d-\ell)$-connected through
codimension one.  In other words, its facet-ridge hypergraph is $(d-\ell)$-connected.
\end{theorem}

The hypotheses on the field, except for the characteristic assumption,
come from the case $d=1$, which is \cite{CartwrightPayne}*{Theorem
  1}.  The characteristic assumption comes from the toric Bertini theorem.
The lineality space depends on the choice of polyhedral complex
structure $\Sigma$ on the tropical variety, as does the facet-ridge
hypergraph.  However the result holds for any choice of polyhedral
structure, whether or not the lineality space is 
the largest possible one; see Example~\ref{e:lineality}.

This result is sharp.  In a $d$-dimensional polyhedral fan with
$\ell$-dimensional lineality space, a simplicial facet has $d-\ell$
ridges as its boundary, so we can remove $d-\ell$ neighboring facets
to isolate the simplicial facet.

Theorem~\ref{t:maintheorem} can be considered a generalization of
Balinski's Theorem, which says that the edge graph of a
$d$-dimensional polytope is $d$-connected \cite{Balinski}.  Any
complete fan is the tropicalization of the algebraic torus $(K^*)^n$,
which is irreducible. If $\Sigma$ is the (complete) normal fan of a
full dimensional polytope $P$ in $\RR^d$, then, since
Theorem~\ref{t:maintheorem} holds for any such fan $\Sigma$, it can be
translated as: removing $d-1$ vertices and their incident edges from
the edge graph of the polytope $P$ does not disconnect this graph.
This is precisely Balinski's Theorem.

A new polyhedral consequence of Theorem~\ref{t:maintheorem} is the following
necessary condition for a $k$-dimensional rational fan in $\mathbb R^n$ to be
the $k$-skeleton of the normal fan of a polytope.

\begin{corollary}
  \label{cor:skeleton}
The $k$-dimensional skeleton of the normal fan of a rational full-dimensional
polytope is $k$-connected through codimension one.
\end{corollary}

\begin{proof}
 The $k$-skeleton of the normal fan of a rational full-dimensional
 polytope $P \subset \mathbb R^n$ is the tropicalization of an
 irreducible variety, as it is the tropicalization of any sufficiently
 general complete intersection of $n-k$ polynomials with Newton
 polytope $P$ \cite[Corollary 4.6.11]{TropicalBook},
 \cite{Yu_generic}.
\end{proof}

When the incidence hypergraph is replaced by the graph in which each
hyperedge is replaced by the clique joining the vertices of the hyperedge,
even higher connectivity holds~\cite{Ath, Sallee}.  However, as noted
above, the hypergraph version of the statement cannot be strengthened.

The connectedness result also gives a new obstruction for the realizability
of a polyhedral complex as the tropicalization of an irreducible
variety, as we will now demonstrate with a simple example.
\begin{example}
\label{ex:notconnected}
Consider the two-dimensional fan in $\RR^5$ that is the union of two
standard tropical planes living in $\spann
\{\mathbf{e}_1,\mathbf{e}_2,\mathbf{e}_3\}$ and $\spann
\{\mathbf{e}_1,\mathbf{e}_4,\mathbf{e}_5\}$ respectively, meeting
along the ray spanned by $\mathbf{e}_1$, where $\mathbf{e}_1,\dots,\mathbf{e}_5$ is the
standard basis for $\mathbb R^5$.
See Figure~\ref{fig:twoK4}.  Each node in the figure represents a ray
and each edge represents a two-dimensional cone of the fan.
The ray spanned by $\mathbf{e}_1$ is the only intersection of these two tropical planes.  Removing a 
closed maximal cone containing $\mathbf{e}_1$ disconnects the fan, so the fan
is not $2$-connected.  It is therefore not the tropicalization of an
irreducible variety over characteristic $0$.

\begin{figure}
\begin{center}
\definecolor{uuuuuu}{rgb}{0.26666666666666666,0.26666666666666666,0.26666666666666666}
\vspace{-0.5in}
\begin{tikzpicture}[line cap=round,line join=round,>=triangle 45,x=1.0cm,y=1.0cm]
\clip(-5.38,-5.37) rectangle (8.66,3.53);
\draw [line width=0.4pt] (0.,0.)-- (-2.,1.1);
\draw [line width=0.4pt] (0.,0.)-- (-2.,-1.18);
\draw [line width=0.4pt] (0.,0.)-- (-1.28,0.);
\draw [line width=0.4pt] (-1.28,0.)-- (-2.,1.1);
\draw [line width=0.4pt] (-1.28,0.)-- (-2.,-1.18);
\draw [line width=0.4pt] (-2.,1.1)-- (-2.,-1.18);
\draw [line width=0.4pt] (0.,0.)-- (1.36,0.);
\draw [line width=0.4pt] (0.,0.)-- (2.02,-1.21);
\draw [line width=0.4pt] (0.,0.)-- (2.02,1.21);
\draw [line width=0.4pt] (1.36,0.)-- (2.02,1.21);
\draw [line width=0.4pt] (1.36,0.)-- (2.02,-1.21);
\draw [line width=0.4pt] (2.02,1.21)-- (2.02,-1.21);
\draw [fill=uuuuuu] (0.,0.) circle (1.0pt);
\draw[color=uuuuuu] (-0.01,0.49) node (e1) {$e_1$};
\draw [fill=uuuuuu] (-1.28,0.) circle (1.0pt);
\draw[color=uuuuuu] (-1.28,0.) node (e123) {};
\draw[color=uuuuuu] (-3.5,-0) node (e123label) {$-e_1-e_2-e_3$};
\draw [fill=uuuuuu] (-2.,1.1) circle (1.0pt);
\draw[color=uuuuuu] (-2.4,1.23) node (e2) {$e_2$};
\draw [fill=uuuuuu] (-2.,-1.18) circle (1.0pt);
\draw[color=uuuuuu] (-2.4,-1.13) node (e3) {$e_3$};
\draw [fill=uuuuuu] (1.36,0.) circle (1.0pt);
\draw[color=uuuuuu] (1.36,0.) node (e145) {};
\draw[color=uuuuuu] (3.5,0) node (e145label) {$-e_1-e_4-e_5$};
\draw [fill=uuuuuu] (2.02,1.21) circle (1.0pt);
\draw[color=uuuuuu] (2.41,1.35) node (e4) {$e_4$};
\draw [fill=uuuuuu] (2.02,-1.21) circle (1.0pt);
\draw[color=uuuuuu] (2.37,-1.17) node (e5) {$e_5$};
\draw[->](e145label) edge[bend right] (e145);
\draw[<-](e123) edge[bend right] (e123label);
\end{tikzpicture}
\vspace{-1.5in}
\end{center}
\caption{The two-dimensional tropical variety from
  Example~\ref{ex:notconnected} depicted here is not $2$-connected
  since removing any facet containing $e_1$ disconnects it.  So it is not the tropicalization of an irreducible variety.}
\label{fig:twoK4}
\end{figure}
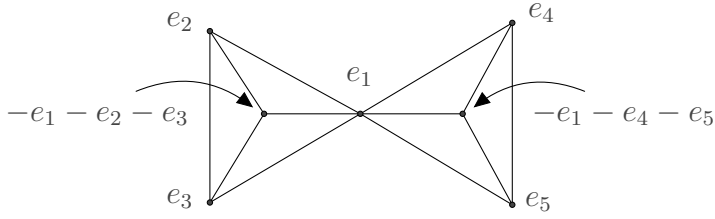

\end{example}

The connectivity depends on the choice of polyhedral structure, and in
particular on the choice of lineality space.

\begin{example} \label{e:lineality}
Let $X$ be the subvariety of $(K^*)^3$ defined by $x_1+x_2+1=0$.  The
tropicalization $\trop(X)$ is the support of the two-dimensional fan
with lineality space $\spann(\mathbf{e}_3)$, and three cones generated
by $\mathbf{e}_1$, $\mathbf{e}_2$, and $-\mathbf{e}_1-\mathbf{e}_2$ in
addition to the lineality space.  With that choice of polyhedral
complex, the facet-ridge hypergraph has three vertices and one
hyperedge containing all vertices, which is connected but not
$2$-connected.  An alternate polyhedral structure is given by
subdividing each cone using the hyperplane $x_3=0$, to obtain a fan with trivial
lineality space and six maximal cones:
$\pos(\mathbf{e}_1,\mathbf{e}_3), \pos(\mathbf{e}_1,-\mathbf{e}_3),
\pos(\mathbf{e}_2,\mathbf{e}_3), \pos(\mathbf{e}_2,-\mathbf{e}_3)$,
$\pos(-\mathbf{e}_1-\mathbf{e}_2,\mathbf{e}_3),$ and
$\pos(-\mathbf{e}_1-\mathbf{e}_2,-\mathbf{e}_3)$.  The facet-ridge
hypergraph now has six vertices, two hyperedges with three vertices
each, and three hyperedges with two vertices each.  This is
$2$-connected.
\end{example}

The fact that tropicalizations of irreducible varieties are connected
through codimension one plays a crucial role in algorithms to compute
tropical varieties, as implemented in \cite{gfan}, following
algorithms originally outlined in \cite{BJSST}.  It would be
interesting to exploit this higher connectivity algorithmically.  As
the space of phylogenetic trees 
can be realized as the
tropicalization of an irreducible variety (the complement of the
type-$A$ hyperplane arrangement) \cite{ArdilaKlivans}, it would be interesting to explore
this connection.  It would also be interesting to understand the
extent to which $d$-connectivity is a combinatorial notion.  In that direction, in
Proposition~\ref{p:matroid} we show that the fine fan structure on the
Bergman fan of a matroid of rank $d+1$ is $d$-connected.  Such Bergman
fans play the role of building blocks for abstract tropical varieties.

The proof of Theorem~\ref{t:maintheorem} involves aspects of both algebraic and polyhedral geometry.  In
Section~\ref{sec:Bertini} we prove the following tropical version of Bertini's
Theorem.
\begin{theorem}[Tropical Bertini Theorem]
\label{thm:Bertini}
Let $X \subseteq (K^*)^n$ be an irreducible $d$-dimensional variety, with $d \geq 2$,
over an algebraically closed valued field $K$ of characteristic $0$ with
$\mathbb Q$ contained in the value group.  The space of rational
affine hyperplanes in $\mathbb R^n$ can be identified with $\mathbb
P^{n}_{\mathbb Q}$.  The set of hyperplanes $H$ for which the
intersection $\trop(X) \cap H$ is the tropicalization of an
irreducible variety 
is dense in the Euclidean topology on $\mathbb P^n_{\mathbb Q}$.
\end{theorem}
 This gives an affirmative answer to Question 9
of~\cite{CartwrightPayne} in characteristic zero, which asks, for an irreducible variety
$X$, whether there always exists an affine hyperplane whose
intersection with $\trop(X)$ is proper and connected through
codimension one.  

In Section~\ref{sec:proof}
we prove the connectivity theorem using the tropical Bertini theorem
and induction on dimension.  Proposition~\ref{p:matroid} on Bergman
fans is proved in Section~\ref{sec:Bergman}, and we state some open
problems in Section~\ref{sec:questions}.

After this paper was posted on the ArXiv, \cite{HY} gave a purely combinatorial proof of
Corollary~\ref{cor:skeleton}, removing the rationality assumption, and \cite{WICA} removed the
characteristic $p$ assumption from the tropical Bertini theorem, and thus,
by Remark~\ref{rem:char}, from Theorem~\ref{t:maintheorem}.

\section{Proof of the  Tropical Bertini Theorem}
\label{sec:Bertini}

Let $K$ be a (possibly trivially) valued field.
The tropicalization of a variety $X \subset (K^*)^n$ is the set
$\trop(X) = 
\overline{\val(X(L))} \subseteq \mathbb R^n$, where $L$ is any
nontrivially valued algebraically closed field extension of $K$, and
the closure is taken in the Euclidean topology.  For example, the
tropicalization of a subtorus of $(K^*)^n$ is a usual linear subspace
of $\RR^n$ defined over $\QQ$.
Taking field extensions does not change the tropicalization. 

The Tropical Bertini Theorem (Theorem~\ref{thm:Bertini}) says that most intersections of
the tropicalization of an irreducible variety and a hyperplane are again the
tropicalization of an irreducible variety.  

\begin{example} \label{e:tropicalbertini}
 Let $X = V(x_1-x_2^2x_3^2) \subseteq \mathbb (K^*)^3$.  The
 intersection of $X$ with the subtorus defined by $x_1^{a_1} x_2^{2a_2}
 x_3^{2 a_3} = 1$ is not irreducible for any tuple of integers
 $(a_1,a_2,a_3) \notin \ZZ(1,-1,-1)$, so it is not the case that the intersection of $X$ with
 a ``generic'' subtorus is irreducible.  However the tropicalization
 of the statement is true.
\end{example}

\begin{remark} 
\label{rem:scale} If a polyhedral complex is the
tropicalization of an irreducible  variety, then so is its image under any invertible linear transformation
over $\QQ$.  To see this, first note that for any nonzero rational
number $r$, the scaling $r \cdot \trop(X)$ is achieved by scaling the
valuation of the coefficient field.  This means that we may scale the
image to assume that the linear transformation is given by an integer
matrix.  Changing basis over $\ZZ$ for the source or target
corresponds to an automorphism of the respective torus, so we only need to
consider diagonal linear transformations.   Finally, note that if $X \subset (K^*)^n$ is
irreducible, so is its image under the map $(x_1,\dots,x_n) \mapsto
(x_1^{c_1},x_2^{c_2}, \dots, x_n^{c_n})$ for any integers
$c_1,\dots,c_n$. 
\end{remark}

The key idea for the proof of Theorem~\ref{thm:Bertini} is the
following subtorus analogue of Bertini's Theorem due to
\cites{Zannier,FuchsMantovaZannier}.  We also make crucial use of some
modifications due to \cite{AmorosoSombra}.  Following
\cite{AmorosoSombra} we will say that a map $\pi : X \rightarrow
(K^*)^d$ satisfies the {\em property PB} (pullback) if the pullback
$\lambda^*X := X \times_{(K^*)^d} (K^*)^d$ of $X$ along $\pi$ and
$\lambda$ is irreducible for every isogeny $\lambda$ of $(K^*)^d$.  An
isogeny of $(K^*)^d$ is a surjective group homomorphism with finite
kernel, so can be represented by a $d \times d$ rank $d$ matrix with
integer entries.

\begin{theorem}[Toric Bertini Theorem. Fuchs, Mantova, Zannier, Theorem 1.5 of \cite{FuchsMantovaZannier}] \label{t:FMZ}~\\
Let $X$ be an irreducible quasiprojective variety of dimension $d$
over an algebraically closed field of characteristic zero, and let
$\pi : X \rightarrow (K^*)^d$ be a dominant map that is finite onto
its image, satisfying property PB.  Then there is a finite union $\mathcal{E}$ of
proper subtori of $(K^*)^d$ such that, for every subtorus $T \subset
(K^*)^d$ not contained in $\mathcal{E}$ and every point $p \in
(K^*)^d$, the preimage $\pi^{-1}(p \cdot T)$ is an irreducible
subvariety of $X$.
\end{theorem}

We also make use of the following lemma, which is a special case of \cite{PayneCorrection}*{Proposition 4}.

\begin{lemma} \label{l:finite} Let $X$ be a $d$-dimensional subvariety of $(K^*)^n$,
and let $\pi: (K^*)^n \rightarrow (K^*)^d$ be a morphism.  If the linear map
$\trop(\pi) \colon \mathbb R^n \rightarrow \mathbb R^d$ is injective
on every maximal face of $\trop(X)$, then the restriction $\pi|_X \colon X \rightarrow
(K^*)^d$ is a finite morphism.  \end{lemma}

\begin{proof}[Proof of Lemma \ref{l:finite}] The tropicalization
  $\trop(X)_{\mathrm{triv}}$ of $X$ with respect to the trivial valuation is the recession fan of $\trop(X)$, so the hypothesis that
$\trop(\pi)$ is injective on every maximal face of $\trop(X)$ 
 implies that the same is true for $\trop(X)_{\mathrm{triv}}$.  When $K$ is given the
trivial valuation, the ``tilted group ring'' $R[M]^\mathbf{0}$ of
\cite{PayneFibers}*{\S 2.4} is just the Laurent polynomial ring
$K[M]$, so for $v = \mathbf{0}$, the scheme $T'_{\phi(v)}$ of
\cite{PayneCorrection}*{Proposition 4} is $(K^*)^d$. The preimage
of $\mathbf{0} \in \mathbb R^d$ under the map
$\trop(X)_{\mathrm{triv}} \rightarrow \mathbb R^d$ is just
$\mathbf{0}$, so the scheme $\mathcal X_{\mathbf{0}}$ of
\cite{PayneCorrection} is just $X$.  Proposition 4 of
\cite{PayneCorrection} then states that the morphism
$\mathcal{X}_{\mathbf{0}} \rightarrow (K^*)^d$ is finite, which proves the lemma.
\end{proof}

\begin{proof}[Proof of Theorem~\ref{thm:Bertini}]
Fix a rational polyhedral complex $\Sigma$ with support $\trop(X)$.
Choose a linear map $P \colon \mathbb{R}^n \rightarrow \mathbb{R}^d$,
given by a $d \times n$ integer matrix, that is injective on every
maximal face of $\trop(X)$.  The kernel of such a map is an
$(n-d)$-dimensional subspace of $\mathbb Q^n$, and the set of such
subspaces is
dense in the analytic topology on the Grassmannian $\Gr(n-d,n)$. 
 We will show that there is a dense open 
set in the space of rational
 affine hyperplanes $\overline{H}$ in $\mathbb R^d$ for which
 $P^{-1}(\overline{H}) \cap \trop(X)$ is the tropicalization of an
 irreducible variety.  
 The set of $H = P^{-1}(\overline{H})$ for all such $P$
 is
dense 
in $ \mathbb
P^{n}_{\mathbb Q}$, proving the Theorem.

Let $\pi : (K^*)^n \rightarrow (K^*)^d$ be the morphism
of tori corresponding to $P$.
 By Remark~\ref{rem:scale} we can change
coordinates to assume that the linear map $P : \RR^n \rightarrow
\RR^d$ is the coordinate projection onto the first $d$ coordinates.
Thus $\pi$ is also a coordinate projection from $(K^*)^n$ onto the
first $d$ coordinates.  By Lemma~\ref{l:finite} the morphism $\pi|_X$ is finite.

In order to apply Theorem~\ref{t:FMZ}, we would like $\pi|_X$ to have
property PB.  
If $\pi|_X$ does not have property PB, then there is an isogeny $\mu: (K^*)^d \rightarrow (K^*)^d$ such that $\mu^* X$ is not irreducible.
  Let $I \subseteq K[x_1^{\pm
    1},\dots,x_n^{\pm 1}]$ be the ideal of $X$, and let $\mu$ be
given by 
\[\mu^* \colon K[z_1^{\pm 1},\dots,z_d^{\pm 1}]
 \rightarrow K[z_1^{\pm 1},\dots,z_d^{\pm 1}] \text{ with }\mu^{*}(z_i) =
 z^{\alpha_i}.\]
 
 Since $\pi$ is assumed to be the projection onto the first $d$
  coordinates, the coordinate ring of the pullback $\mu^* X$ is
  \[K[x_1^{\pm 1},\dots,x_n^{\pm 1},z_1^{\pm 1},\dots,z_d^{\pm 1}]/ (I
  + \langle x_i - z^{\alpha_i} : 1 \leq i \leq d \rangle ).\]
This is isomorphic to \[  K[x_{d+1}^{\pm 1},\dots,x_n^{\pm 1},z_1^{\pm 1},\dots,z_d^{\pm
      1}]/J\]
 where $J$ is obtained from $I$ by substituting
 $z^{\alpha}_i$ for $x_i$ for $i=1,\dots,d$.
Write $\rho \colon (K^*)^n \rightarrow (K^*)^d$ for the map
corresponding to the inclusion $K[z_1^{\pm 1},\dots,z_d^{\pm 1}]
\rightarrow K[x_{d+1}^{\pm 1},\dots,x_n^{\pm 1},z_1^{\pm
    1},\dots,z_d^{\pm 1}]$.

Write $Z$ for the reduced structure on $\mu^*X$, which is described by
the radical of the ideal $J$.  We have the following diagram:

\begin{equation}
\label{eqn:diag1}
\xymatrix{ Z \subseteq (K^*)^{d} \times (K^*)^{n-d} \ar[r] \ar[dr]_{\rho} & \mu^*X \subseteq (K^*)^d \times (K^*)^{n-d} \ar[rrr]^{\mu'
    =(\mu, \mathop{id}_{n-d})} \ar[d] & & & X \subseteq (K^*)^n
  \ar[d]^{\pi} \\ &(K^*)^d \ar[rrr]^{\mu}& & & (K^*)^d \\ }.
\end{equation}

The components of $Z$ have dimension at least $\dim(X)$, by the
Principal Ideal Theorem (see for example
\cite{Eisenbud}*{Theorem 10.2}).  Since the pullback
$\mu' \colon \mu^*X \rightarrow X$ of the finite morphism $\mu$ is
finite, this morphism is also proper, so the image of every
irreducible component is closed.  As $X$ is irreducible, there
must thus be a component $Y$ that maps surjectively to $X$.

The group $\ker(\mu)$ acts on $Z$ via
$(t,x) \mapsto (\zeta \cdot t, x)$ for $\zeta \in \ker(\mu)$.  This
group acts transitively on the fibers of $\mu$, so the orbit of $Y$ is all of $Z$.
We also have that $Y$ is not fixed by
this action, as otherwise it would be the reduced structure on the
only irreducible component of $\mu^*X$, contradicting the assumption
that $\mu^*X$ is reducible.  This means that the map $\mu'|_Y$ has
degree less than $\deg(\mu)$.  

If $\mu$ is chosen so that $\rho|_Y$ has minimal degree,
then $\rho|_Y$ must have property PB.  Otherwise we can repeat the
above argument with $Y$ and $\rho$ in the place of $X$ and $\pi$ to
further reduce the degree since
\[  \deg(\rho|_Y) = \deg(\rho|_Y)\deg(\mu)/\deg(\mu) = \deg(\pi|_X) \deg(\mu'|_Y)/\deg(\mu) <  \deg(\pi|_X).\]

The isogeny $\mu : (K^*)^d \rightarrow (K^*)^d$ has a
natural tropicalization $\trop(\mu) : \RR^d \rightarrow \RR^d$ such
that $\trop \circ \mu = \trop(\mu) \circ \trop$ as maps from $(K^*)^d$
to $\RR^d$.
Then Diagram~\eqref{eqn:diag1} tropicalizes as 
\begin{equation}
\label{eqn:commdiag}
\xymatrix{ \trop(Y) \subseteq \mathbb R^n \ar[rrr]^{\trop(\mu')}
  \ar[d]_{\trop(\rho)} & & & \trop(X) \subseteq \mathbb R^n
  \ar[d]^{\trop(\pi) = P} \\ \mathbb R^d \ar[rrr]^{\trop(\mu)}& & &
  \mathbb R^d \\ }
\end{equation}
where $\trop(\mu')$ maps $\trop(Y)$ to $\trop(X)$.

Since $\trop(\mu)$ and $\trop(\mu')$ are invertible linear maps, 
the assumption that $P$ is injective on every maximal face of $\trop(X)$
implies that $\trop(\rho)$ is also injective on every maximal face of
$\trop(Y)$.
  Therefore $\rho|_Y$ is a finite morphism by Lemma~\ref{l:finite}.
We can now apply Theorem~\ref{t:FMZ} to $\rho|_Y \colon Y \rightarrow (K^*)^d$.  Let $\mathcal E$ be the exceptional set in the
theorem, which is a union of finitely many proper subtori.  Then
$\trop(\mathcal E)$ is the union of finitely many proper linear subspaces of $\mathbb R^d$.  
 The affine rational hyperplanes that are not parallel to any subspace in $\trop(\mathcal E)$
form a dense open subset of $\mathbb{P}^d_\QQ$.
Let $\overline{H}$ be one such generic hyperplane.
Let $T \subseteq (K^*)^d$ be the torus such that $\trop(T)$ is a
 hyperplane parallel to $\overline{H}$.  
 For any $s \in (K^*)^d$ with $\val(s) \in \overline{H}$, we have $\trop(s \cdot
T) = \overline{H}$.  By  Theorem~\ref{t:FMZ},
$\rho|_Y^{-1}(s \cdot T) \subset Y$ is irreducible.  

Let $H$ be the
hyperplane in $\RR^n$ defined by $H := (\trop
\rho)^{-1}(\overline{H})$.  Since $\rho$ is a monomial map from
$(K^*)^n \rightarrow (K^*)^d$, we have $H = \trop(\rho^{-1}(s \cdot T))$.  
The intersection $H \cap \trop(Y)$ is transverse because $\trop(\rho)$ is injective on every maximal face of $\trop(Y)$.
Thus the Transverse Intersection
Lemma~\cite{OssermanPayne}*{Theorem~1.1}, \cite{BJSST}*{Lemma 15},
\cite{TropicalBook}*{Theorem 3.4.12} implies that 
\[H \cap \trop(Y) = \trop(\rho^{-1}(s \cdot T) \cap Y) = \trop(\rho|_Y^{-1}(s \cdot T)) .\]
Thus $H \cap \trop(Y)$ is the tropicalization of the irreducible
variety $\rho|_Y^{-1}(s \cdot T)$.  It follows from Remark~\ref{rem:scale} that the same is true
for $\trop(\mu')(H) \cap \Sigma =
\trop(\pi)^{-1}(\trop(\mu)(\overline{H}))$.  As the set of
$\trop(\mu)(\overline{H})$ where $\overline{H}$ is not parallel to
any subspace in $\trop(\mathcal E)$ is dense and open in $\mathbb P^d_{\mathbb Q}$,
the result follows.
\end{proof}

\begin{example}
We illustrate the key constructions of the proof on the variety from Example~\ref{e:tropicalbertini}: $X = V(x_1-x_2^2x_3^2) \subseteq \mathbb (K^*)^3$.
The tropicalization is $\trop(X) = \{\mathbf{w} \in \mathbb R^3 : w_1 =2w_2+2w_3 \}$.  The projection of
$\trop(X)$ to the first two coordinates is a bijection, so the tropicalization
of the projection $\pi \colon \mathbb (K^*)^3
\rightarrow \mathbb (K^*)^2, (x_1,x_2,x_3) \mapsto (x_1, x_2)$ satisfies the condition of the first
paragraph of the proof.  

Consider the isogeny $\mu : (K^*)^2
\rightarrow (K^*)^2$ given by $\mu(z_1,z_2) = (z_1^4,z_2)$.  The coordinate ring of the  pullback $\mu^* X$ is
\[K[x_1^{\pm 1},x_2^{\pm
    1},x_3^{\pm 1},z_1^{\pm 1},z_2^{\pm 2}]/\langle x_1-x_2^2x_3^2,
x_1-z_1^4,x_2-z_2 \rangle \cong K[x_3^{\pm 1},z_1^{\pm 1},z_2^{\pm
    1}]/\langle z_1^4-z_2^2x_3^2 \rangle,\]
so the ideal $J$ of the
proof is $\langle z_1^4-z_2^2x_3^2 \rangle \subseteq K[z_1^{\pm
    1},z_2^{\pm 1},x_3^{\pm 1}]$.  Note that $V(J)$ is not irreducible, so $\pi$ does not have property PB.
It is, however, radical, so defines the subscheme $Z$.  The map $\mu' \colon Z \rightarrow X$ given by 
$\mu' \colon (z_1,z_2,x_3) \mapsto (z_1^4,z_2,x_3)$ has degree $4$.

One component of $Z$ is given by $Y = V(z_1^2-z_2x_3)$.  This maps
surjectively to $X$ via the map $\mu'$, and $\mu'|_Y$ has degree $2$.
The kernel of $\mu$ is the multiplicative abelian
group $\{(1,1),(-1,1),(i,1),(-i,1)\}$.

We now check that the projection of $Y$ onto the first two coordinates
has property PB.  For every isogeny $\lambda$ on $(K^*)^2$, given by
$\lambda(t_1, t_2)= (t_1^{a_1} t_2^{a_2}, t_1^{b_1}t_2^{b_2})$, the
coordinate ring of the pullback $\lambda^* Y$ is
\[
K[z_1^{\pm 1}, z_2^{\pm 1}, x_3^{\pm 1}, t_1^{\pm 1}, t_2^{\pm 1}] /
\langle z_1^2 - z_2 x_3, z_1 - t_1^{a_1} t_2^{a_2}, z_2 - t_1^{b_1}t_2^{b_2}
\rangle \cong K[x_3^{\pm 1}, t_1^{\pm 1}, t_2^{\pm 1}] /
\langle  t_1^{2a_1} t_2^{2a_2} -  t_1^{b_1}t_2^{b_2} x_3 \rangle
\] 
The binomial $t_1^{2a_1} t_2^{2a_2} -  t_1^{b_1}t_2^{b_2} x_3$ is
irreducible in the Laurent polynomial ring, so the pullback is irreducible.
\end{example}

\begin{remark}
\label{rem:char}
The results in the paper \cite{FuchsMantovaZannier} are stated over
$\mathbb C$.  The proofs there rely on resolution of
singularities and on some facts from Galois theory that are simpler
in characteristic zero, but all go through in the case that the field
is algebraically closed of characteristic zero.  
The Tropical Bertini Theorem (Theorem~\ref{thm:Bertini}) is the only
place where the characteristic zero assumption is used in the proof of
the higher connectivity result in Theorem~\ref{t:maintheorem}.
\end{remark}

\section{Proof of the Connectivity Theorem}
\label{sec:proof}
We now prove the main result of this paper: Theorem~\ref{t:maintheorem}.

Let $X$ be a $d$-dimensional irreducible subvariety of $(K^*)^n$,
where $K$ is a field of characteristic zero that is either
algebraically closed, complete, or real closed with convex valuation
ring.  Fix a polyhedral complex $\Sigma$ with support
$|\Sigma|$ equal to $\trop(X)$, such that the normal fan of every
facet of $\Sigma$ is a rational polyhedral fan.

  Suppose $\Sigma$ has an $\ell$-dimensional lineality space $V$.  The
  rationality assumption implies that $V$ is a rational subspace of
  $\mathbb R^n$, so it determines an $\ell$-dimensional subtorus $T'$
  of $(K^*)^n$, with $\trop(T')=V$.  We claim that $T'$ acts on $X$.
  If $T'$ did not act on $X$, then the orbit of $X$ under $T'$ would
  have dimension greater than $d$, and its tropicalization would also
  have dimension greater than $d$.  On the other hand, the
  tropicalization of this orbit is $\Sigma + V = \Sigma$, which has
  dimension $d$.  Therefore $T'$ must act on $X$.  Since the action of
  $T'$ on $(K^*)^n$ is free, the restriction to $X$ is also free.  Let
  $\widetilde{X} = X/T' \subseteq (K^*)^n/T' \cong (K^*)^{n-\ell}$ be
  the quotient.  Then $\widetilde{X}$ is irreducible of dimension
  $d-\ell$, $\trop(\widetilde{X})= \Sigma/V$ is the quotient of
  $\trop(X)$ by its lineality space, and $\Sigma$ and $\Sigma/V$ have
  the same connectivity.  Therefore we may reduce to the case when the
  lineality space is trivial and $\ell=0$.  Since $\trop(X)$ is
  connected, the triviality of the lineality space implies that every
  face of $\trop(X)$ is pointed, meaning that it does not contain an
  affine line.

By \cite{CartwrightPayne}*{Proposition 4} we may replace the field $K$
by its algebraic closure, and $X$ by any irreducible component of the
extension $X_{\overline{K}}$ without changing the tropicalization.  We
henceforth assume that $K$ is algebraically closed, and $X$ is
geometrically irreducible.  We may thus pass to a further field
extension if necessary to assume that the valuation on $K$ is
nontrivial, and $\mathbb Q$ is contained in the value group.

We may also assume that $\trop(X)$ does not live in any proper
subspace in $\mathbb R^n$.  Otherwise $X$ lives in a translate of a
subtorus of $(K^*)^n$ by~\cite{JensenKahleKatthaen}*{Corollary~14},
and the result for the translate of $X$ in the subtorus implies the
result for $X$.

We will first prove the result when the tropical variety is all of
$\RR^d$, with an arbitrary polyhedral complex structure but trivial
lineality space.  It suffices to show that for any $d \geq 1$ removing $d-1$
closed pointed convex sets $C_1,\dots,C_{d-1}$ from $\RR^d$ does not
disconnect $\RR^d$.  Indeed, in our context this would imply that
there was a path between any two points in the interior of  facets.
If this path passed through a codimension two cell, none of the
neighborhood of this cell would have been removed, so we can deform the path to only pass through facets and codimension-one cells.

To show connectedness, identify $\RR^d$ with the hyperplane defined by $x_1 = 1$ in
$\RR^{d+1}$.  Let $S^d$ be a sphere around the origin in $\RR^{d+1}$.
Let $\widetilde{C_i}$ be the closure of $\mathrm{cone}(\{1\} \times
C_i) \cap S^d$, and let $H = \{x \in S^d \mid x_1 \leq 0\}$ be the
lower hemisphere.  Then $S^d \setminus (H \cup \widetilde{C_1} \cup
\cdots \cup \widetilde{C_{d-1}})$ is homeomorphic to $\RR^d \setminus
(C_1\cup \cdots \cup C_{d-1})$.

By Alexander duality, $S^d \setminus (H \cup \widetilde{C_1} \cup
\cdots \cup \widetilde{C_{d-1}})$ is connected if and only if $H \cup
\widetilde{C_1} \cup \cdots \cup \widetilde{C_{d-1}}$ does not have
homology in dimension $d-1$.  Since the $C_i$s are pointed and convex,
the intersection of any collection of $\widetilde{C_i}$s is
contractible.  The intersection of any of them with the lower
hemisphere is also contractible.  By the Nerve Theorem for regular
cell complexes~\cite{Bjorner}*{Theorem~10.6} the union $H \cup
\widetilde{C_1} \cup \cdots \cup \widetilde{C_{d-1}}$ has the same
homology as its nerve complex.  But the nerve complex is a simplicial
complex on $d$ vertices, so it cannot have homology in dimension
$d-1$.  This completes the proof when the tropical variety is all of
$\RR^n$.

We assume from now on that $|\Sigma| \neq \RR^n$. 
We proceed by
 induction on the dimension $d$ of~$\Sigma$.  
 The base case $d=1$ is the statement that the tropicalization of an
 irreducible curve is connected (see \cite{CartwrightPayne} or
 \cite{TropicalBook}*{Proposition~6.6.22}); this result does not depend
 on the choice of polyhedral complex structure on $\trop(X)$.  The
 hypothesis on the field $K$ is used here.

Now let $d \geq 2$ and assume that the assertion in the theorem is
true for smaller dimensions. Let $\mathcal G$ be any collection of
$d-1$ facets of $\Sigma$.  For any two facets $P, Q$ of $\Sigma
\setminus \cG$ we need to show that after removing the closed facets
in $\mathcal G$, there is still a path from $P$ and $Q$ through ridges
and facets only.

Let $F$ be any facet in $\cG$.  We define an equivalence relation $\sim_F$ on the
facets of $\Sigma \setminus F$ as the transitive closure of the
following relation: $P \sim_F Q$ if there exists a hyperplane $H$ with
$H \cap F = \varnothing$, which meets both $P$ and $Q$ in their
relative interior. 

The theorem then follows from the following two claims.

\begin{claim}\label{lem:path}
For facets $P,Q$ of $\Sigma$ which are not in $\cG$, 
if $P \sim_F  Q$, then there is a facet-ridge path between $P$ and $Q$ that avoids $\cG$.
  \end{claim}

\begin{claim}\label{lem:oneClass}
The equivalence relation $\sim_F$ has only one equivalence class.\\ \qed
\end{claim}

\begin{proof}[Proof of Claim~\ref{lem:path}]
Suppose that $P \sim_F Q$.  We first observe that there is a rational hyperplane $H$ such that
\begin{enumerate}
\item[(a)] $H$ does not meet $F$ but meets the relative interiors of
  both $P$ and $Q$,
\item[(b)] $H$ does not contain any vertex of $\Sigma$, and
\item[(c)] $H \cap \Sigma$ is the tropicalization of an irreducible
  variety.
\end{enumerate}

The assumption that $P \sim_F Q$ implies the existence of some
hyperplane $H$ not intersecting $F$ but meeting the relative interiors
of $P$ and $Q$.  Since $\dim(\Sigma)\geq 2$,
$H$ can be perturbed to be rational, and satisfy (a) and (b). Moreover
further small perturbation will preserve (a) and (b). By
Theorem~\ref{thm:Bertini} the hyperplanes satisfying (c) are dense in
$\mathbb P^{n}_{\mathbb{Q}}$.  
Thus there exists a rational hyperplane $H$ satisfying (a),(b), and
(c).

Let $Z$ be an irreducible variety such that $\trop(Z) = H \cap \Sigma$.
By induction on dimension, $\trop(Z)$ is $(d-1)$-connected through
codimension one.  The condition (b) above implies that every facet of $\trop(Z)$ is the intersection of a facet of $\Sigma$ with $H$, and every ridge of $\trop(Z)$ is the intersection of a ridge of $\Sigma$ with $H$.  This gives a natural inclusion of the facet-ridge hypergraph of $\trop(Z)$ into the facet-ridge hypergraph $G$ of $\Sigma$.
By construction the resulting sub-hypergraph $G'$ of $G$ contains the
vertices corresponding to $P$ and $Q$, but not the vertex
corresponding to $F$ or any of its incident hyperedges.  While $G'$ may
contain some of the vertices corresponding to $\mathcal G \setminus \{F\}$, since $G'$ is $(d-1)$-connected, $G'$ with  $\mathcal G \setminus \{F\}$
 and their incident hyperedges removed is still connected, so
there is a path between the vertices corresponding to $P$ and
$Q$.  There is thus a path between $P$ and $Q$ in $G$ not passing
through the vertices corresponding to any element of $\mathcal G$.
\end{proof}

\begin{proof}[Proof of Claim~\ref{lem:oneClass}]

Let $P^\circ$ denote the relative interior of $P$.
For any two facets $P, Q $ of $\Sigma
\setminus \mathcal G$, we claim that if there is a line $L$ which intersects
$P^\circ$ and $Q^\circ$ but does not intersect $F$, then $P \sim_F Q$. To see this, take a 
linear projection onto $\RR^{n-1}$ such that the image of the line $L$
is a point $\overline{L}$ and the image of $F$ is a convex polyhedron
$\overline{F}$.  Since $L$ does not meet $F$, we have $\overline{L} \notin
\overline{F}$.  Let $\overline{V}$ be a hyperplane through $\overline{L}$ which
does not meet $\overline{F}$.  The preimage of $\overline{V}$ in
$\RR^n$ is the desired hyperplane through $P^\circ$ and
$Q^\circ$ which does not meet $F$.  Thus $P \sim_F Q$.

Recall that we are assuming that $\Sigma$ is not full dimensional.
 Fix $F \in \mathcal G$.   Let $H$ be an affine hyperplane containing $F$.
Let $P,Q$ be two facets of $\Sigma \setminus \mathcal G$. 
If there are points $p \in P^\circ \setminus H$ and $q \in Q^\circ \cap H$,
then the line through $p$ and $q$ meets $H$ only at $q$, so it does
not intersect $F$, so  $P
\sim_F Q$ as shown above.  In fact, we can weaken the condition: if
$p \in P \setminus H$ and $q \in (Q \cap H)
\setminus \cG$, then we perturb the line through $p$ and
$q$ to get a line through $P^\circ$ and $Q^\circ$ which does not
intersect $F$, because $F$ is closed, so we still have $P \sim_F Q$.

If  there are facets $P,Q$ in
$\Sigma$ such that $(P \setminus H) \setminus \cG \neq \emptyset$ and $(Q \cap H) \setminus
\cG \neq \emptyset$, then every facet of $\Sigma \setminus \cG$ not
contained in $H$ is equivalent to $Q$ under $\sim_F$: every facet of $\Sigma
\setminus \cG$ contained in $H$ is equivalent to $P$ under $\sim_F$, and
$P \sim_F Q$, so Claim~\ref{lem:oneClass} follows by transitive closure.
In other words, it suffices to show that there is a hyperplane $H
\supset F$ so that $(\Sigma \setminus H) \setminus \cG \neq \emptyset$ and $(\Sigma \cap H) \setminus \cG \neq
\emptyset$.

 We are assuming that $\Sigma$ is not contained in any affine
 subspace.  Let $\sigma$ be a facet of $\Sigma$ which is not contained
 in $H$ for some hyperplane $H$ containing $F$.  Since $\sigma$ is a
 $d$-dimensional pointed polyhedron, it is incident to at least $d$
 distinct ridges $\tau_1,\dots,\tau_d$ of $\Sigma$.  By the balancing
 condition at $\tau_i$ there is at least one other facet $\sigma_i$
 containing $\tau_i$.  If $\tau_i \not\subset H$, then $\sigma_i
 \not\subset H$.  If $\tau_i \subset H$, then we can choose $\sigma_i$
 to be on the other side of $H$ from $\sigma$ by the balancing
 condition.  Thus we have found $d$ distinct facets
 $\sigma_1,\dots,\sigma_d$, none of which is in $H$, so at least one
 of these is not in $\cG$.  Therefore $(\Sigma \setminus H) \setminus
 \cG \neq \emptyset$ for any hyperplane $H$ containing $F$.

If $\Sigma$ has codimension greater than one, then take any point $x
\in \Sigma \setminus \cG$ and let $H$ be a hyperplane containing $x$
and $F$.  Then $x \in (\Sigma \cap H) \setminus \cG \neq \emptyset$.

It thus remains to show this in the case that $\Sigma \subset \RR^n$
is a balanced pointed polyhedral complex of dimension $d = n-1$.  We
will show the following by induction on $r$: if $V$ is an
$r$-dimensional affine subspace of $\RR^n$ such that $\Sigma \cap V
\neq \emptyset$, then $\Sigma \cap V$ is not contained in the union of
$r-1$ facets of $\Sigma$.  The case of $r=1$ is trivial.  Suppose $V$
is an affine subspace of dimension $r \geq 2$ such that $\Sigma \cap V
\neq \emptyset$.  We first claim that $\Sigma \cap V$ is not contained
in a single facet of $\Sigma$.  This is immediate if $V$ contains
$\Sigma$, as a pure positive-dimensional balanced polyhedral complex
with trivial lineality space has more than one facet.  If $\Sigma
\not\subseteq V$, then since $\Sigma \cap V \neq \emptyset$, we have
that $\Sigma \cap V$ contains the stable intersection of $\Sigma$ and
$V$ by \cite{JensenYuStableIntersection}*{Lemma 2.6}.  This stable
intersection is a balanced polyhedral complex of positive dimension by
\cite{JensenYuStableIntersection}*{Theorem 2.13}, so not contained in
a single pointed facet of $\Sigma$.
 If $\Sigma \cap V$ is contained in the union of $r-1$ facets of
 $\Sigma$, then we choose one of these facets $\sigma$, a point $x \in
 \Sigma \cap V \setminus \sigma$, and an affine subspace $V' \subset
 V$ of dimension $r-1$ that contains $x$ but does not intersect
 $\sigma$.  Then $\Sigma \cap V'$ is contained in the union of $r-2$
 facets of $\Sigma$, contradicting the induction hypothesis.  Applying
 this for $r=d$, we see that for any hyperplane $H$, $\Sigma \cap H$
 is not contained in the union of the $d-1$ facets in $\cG$, so
 $(\Sigma \cap H)\setminus \cG \neq \emptyset$.
\end{proof}

\section{Bergman fans}
\label{sec:Bergman}

The trivial valuation tropical variety of a linear space that is given
as the row span of a matrix $A$ depends only on the vector matroid on the columns of $A$.  This tropical variety is called the {\em Bergman fan} of the matroid.
Bergman fans are defined also for matroids that are not
representable over any field.  The Bergman fan of a rank $r$ matroid
is an $r$-dimensional polyhedral fan whose lineality space contains
$\mathbf{1} = (1,1,\dots,1)$.  In~\cite{ArdilaKlivans} Ardila and
Klivans gave two fan structures to a Bergman fan: the {\em coarse subdivision} is
derived from the matroid polytope, and the {\em fine
  subdivision} is a cone over a geometric realization of the order
complex of its lattice of flats.  The rays of the fine
subdivision are the $0/1$ indicator vectors of flats of the matroids
  and the cones correspond to chains of flats ordered by inclusion.
  We will now show that the fine subdivision of the Bergman fan of any
  matroid $M$, representable or not, is $\rk(M)-1$ connected. 

\begin{proposition} \label{p:matroid}
Let $M$ be a matroid of rank $d+1$ on the ground set $\{0,\dots, n\}$,
and let $\Sigma$ be the fine subdivision on the Bergman fan $\trop(M)$
in $\mathbb R^n \cong \mathbb R^{n+1}/\mathbb R
\mathbf{1}$. The facet-ridge hypergraph of $\Sigma$ is $d$-connected.
\end{proposition}

\begin{proof}
The Bergman fan of $M$ with the fine subdivision is a $d$-dimensional pointed
fan in $\mathbb R^n \cong \mathbb R^{n+1}/\mathbb R \mathbf{1}$.  The proof is by
induction on the rank, $d+1$, of the matroid.  When $d=1$, $\Sigma$ is a
one-dimensional fan, with all facets (rays) meeting at the origin, so the fan is connected.  We now assume $d>1$.

Fix a collection $\mathcal G$ of $d-1$ facets of $\Sigma$.  The
facets in $\mathcal G$ are labeled by maximal chains of flats in $M$,
so every facet contains exactly one rank-one flat.
The rank-one flat in each chain is a parallel class of elements
of $M$.  There are at least $d+1$ parallel classes because $M$ has
rank $d+1$.
Since every rank-one flat is in a facet, there are at least $d+1$ facets, so we can fix two
other facets $\sigma_1, \sigma_2 \not \in \mathcal G$.
Since $|\mathcal G|=d-1$, there is an element $i \in
\{0,\dots,n\}$ which is not in any of the rank-one flats of cones in $\mathcal G$.
We write $M_1$ and $M_2$ for the initial matroids corresponding to
vectors in the relative interior of $\sigma_1$ and $\sigma_2$
respectively.  As these cones are in $\trop(M)$, both these matroids
are loop-free, so there are bases $B_1, B_2$ containing~$i$ in
$M_1,M_2$ respectively.

An ordering of a basis $B = (b_1,b_2,\dots,b_{d+1})$ determines a
chain of flats \[
\cl(b_1) \subsetneq \cl(b_1,b_2) \subsetneq \cdots
\subsetneq  \cl(b_1,b_2, \dots, b_{d+1}) = \{0,1,\dots,n\}\]
 in $M$, where the closure $\cl(S)$ denotes the minimal flat
 containing $S$.  Therefore for every basis $B$ of $M$ 
there are at least $d!$ cones of $\Sigma$ with $i$ in their rank-one
flat, and with $B$ as a basis of the corresponding initial matroid.

Choose facets $\sigma'_1, \sigma'_2$ of $\Sigma$ such that $B_1$ and
$B_2$ respectively are bases of the corresponding initial matroid and
$i$ is in their rank-one flat. Thus $\sigma'_1, \sigma'_2$ are not in $\mathcal G$,
 
By \cite{RinconLocalTropical}*{Theorem 2.6} the collection of cones in
$\Sigma$ for which $B$ is a basis of the corresponding initial matroid is homeomorphic to $\mathbb R^d$, and thus is
$d$-connected.  This means that there is a facet-ridge path between
$\sigma_1$ and $\sigma'_1$, and between $\sigma_2$ and $\sigma'_2$,
after the closed facets in $\mathcal G$ are removed. 

 It remains to show that $\sigma'_1$ and $\sigma'_2$ are connected.
 By construction $\sigma'_1$ and $\sigma'_2$ have the same rank-one flat
 $F_i$ containing $i$.  The flats containing $F_i$ (or $i$) are in
 natural inclusion-preserving  bijection with the flats of the contraction $M/i$.  Therefore the star of $\Sigma$ at the ray corresponding to $F_i$ is the Bergman fan of the contraction $M/i$.  Since $M/i$ has rank $d$, by induction its Bergman fan with
the fine subdivision is $(d-1)$-connected, so the star of $F_i$ in
$\Sigma$ is $(d-1)$-connected.

 By assumption, no element of $\mathcal G$ has rank-one flat $F_i$, so there are no elements of $\mathcal G$ in this star.
Thus there is a ridge path in the star of $F_i$
 between $\sigma'_1$ and $\sigma'_2$ avoiding $\mathcal G$, which 
 completes the proof.
\end{proof}

\section{Open Problems}
\label{sec:questions}

\begin{description}
\item[Bergman fans and tropical linear spaces] Is any (nonrealizable) Bergman fan $r-\ell$ connected through codimension one with the coarse fan structure? Any fan structure?  Here $r$ and $\ell$ denote the
  dimension of the Bergman fan and the dimension of the lineality
  space respectively.  Similarly for (nonrealizable) tropical linear spaces of
  valuated matroids as defined in~\cite{Speyer}. 
\item[Connectivity, irreducibility, and algebraic matroids]
An irreducible variety $X$ in $K^n$ gives a matroid on ground set $\{1,2,\dots,n\}$ where the rank of a subset $S \subset \{1,2,\dots,n\}$  is the dimension of the projection of $X$ onto the
$S$-coordinates.  This is called the algebraic matroid of $X$, and
tropicalization preserves algebraic matroids~\cite{YuAlgMatroids}.  In
general it is not known which balanced polyhedral complexes give
matroids this way.
Does the higher connectivity imply that a balanced polyhedral complex
gives an algebraic matroid?  If the balanced complex is  {\em
  tropically irreducible}, in the sense that it is not a union of two
proper balanced weighted rational polyhedral complexes, does the
complex have a higher connectivity or give an algebraic matroid?

\end{description}

\subsection*{\bf Acknowledgements}  This project was revived when the
authors were visiting the Mittag Leffler Institute during the program
on {\em Tropical geometry, Amoebas and Polytopes} in Spring 2018 and
the Institute for Computational and Experimental Research in
Mathematics during the program on {\em Nonlinear Algebra} in Fall 2018, supported by the National Science Foundation under Grant
No. DMS-1439786. They are grateful to the program organizers, and to
the institutes, for the excellent working conditions.  They also thank
Daniel Hathcock, Anders Jensen, Megan Owens, and Raman Sanyal for discussions, and Martin Sombra and
Vincenzo Mantova for answering questions about the toric Bertini
theorems.  JY was partially supported by US National Science
Foundation DMS grants \#1600569 and \#1855726.  DM was partially supported by EPSRC
grant EP/R02300X/1, and also, while at the Mathematical Sciences
Research Institute in Berkeley, California, during April 2019, by NSF
Grant \#1440140.

 \end{document}